\documentclass[12pt]{article}
\usepackage{amssymb,amsmath,amsfonts,amsthm,amsxtra,setspace}
\usepackage{indentfirst}

\newcommand{\mc}{\mathcal}
\newcommand{\bsigma}{\mathbf{\mathop{\pmb{\sum}}}}

\newcommand{\abs}[1]{\left| #1 \right|}

\newcommand{\PPI}{\mathbf{PPI}}

\newcommand{\noopsort}[1]{}

\newtheorem{theorem}{Theorem}[section]

\newsavebox{\Prfref}

\newsavebox{\prfref}

\newtheoremstyle{ref}% an optional reference is inserted after the theorem number
{\topsep}	%      Space above
{\topsep}	%      Space below
{\it}%         		Body font
{}%         		Indent amount (empty = no indent, \parindent = para indent)
{}%			Thm head font
{}%        		Punctuation after thm head
{ }%			Space after thm head: " " = normal interword space;
{\thmname{{\bfseries#1}}\thmnumber{ \textbf{#2\thmnote{\rm #3}\textbf .}}}%Thm head spec (can be left empty, meaning `normal')

\theoremstyle{ref}

\newtheorem{thm}[theorem]{Theorem}
\newtheorem{prop}[theorem]{Proposition}

\newtheorem{cor}[theorem]{Corollary}
\theoremstyle{nnref}
\newtheorem*{defn}{Definition}

\begin{document}
%\begin{titlepage}
\title{PFA$(S)[S]$ for the Masses}
\author{Franklin D. Tall{$^1$}}

\footnotetext[1]{Research supported by NSERC grant A-7354.\vspace*{2pt}}
\date{\today}
\maketitle

\begin{abstract}
We present S. Todorcevic's method of forcing with a coherent Souslin tree over restricted iteration axioms as a black box usable by those who wish to avoid its complexities but still access its power.
\end{abstract}

\renewcommand{\thefootnote}{}
\footnote
{\parbox[1.8em]{\linewidth}{$2000$ Math.\ Subj.\ Class.\ Primary 54A35, 54D15, 54D20, 54D45, 03E35, 03E57; Secondary 54D55, 03E50, 03E55.}\vspace*{5pt}}
\renewcommand{\thefootnote}{}
\footnote
{\parbox[1.8em]{\linewidth}{Key words and phrases: PFA$(S)[S]$, Martin's Axiom, Martin's Maximum, P-ideal Dichotomy,
forcing with a coherent Souslin tree, locally compact normal, Axiom R.}}

\section{Introduction}
This paper is dedicated to Professor Y. Kodama, who hosted my very first lecture in Japan more than 35 years ago, which was the start of many fruitful interchanges with Japanese topologists and set theorists.

This note is an expanded version of a talk presented at the recent \emph{First Pan-Pacific International Conference on Topology and Applications}.  I thank the organizers for inviting me and for a superb conference.  I thank Editor J. E. Vaughan of \emph{Topology and its Applications} for inviting me to write this up for that journal.

Todorcevic invented his method in 2001 \cite{To1} in order to investigate problems not decided by PFA or combinatorial principles such as $\diamondsuit$.  A collection of remarkable results has since been obtained using this method by Todorcevic, P. Larson, A. Dow, and the author.  The proofs are technically difficult, so just as \emph{Martin's Axiom} is accessible to those who do not understand iterated forcing, it would be nice to be able to apply this method without the difficult forcing.  At present, there is no one axiom that will accomplish this, but we can at least list the more important consequences so far of the method, so that they may be further applied.

A \emph{coherent Souslin tree} is a particular kind of very homogeneous Souslin tree.  Its exact definition will not concern us here. All we need know is that its existence follows from $\diamondsuit$ \cite{Larson}.

The notations MA$_{\omega_1}(S)$, PFA$(S)$, MM$(S)$ refer to the weaker versions of MA$_{\omega_1}$, PFA, and MM (Martin's Maximum \cite{FMS}) obtained by restricting only to those posets that preserve the coherent Souslin tree $S$ under countable chain condition, proper, and preserving-stationary-subsets-of-$\omega_1$ forcing, respectively.  Notations such as \emph{PFA$(S)[S]$ implies $\Phi$} are shorthand for ``in any model obtained by forcing with $S$ over a model of PFA$(S)$, $\Phi$ holds".

A heuristic analogy which may be helpful to the reader is to recall that two principal kinds of consequences of MA$_{\omega_1}$ are the ``combinatorial" ones following from MA$_{\omega_1}$($\sigma$-centred), and the ``Souslin-type" ones \cite{KT}.  Souslin-type ones are those that imply Souslin's Hypothesis, such as \emph{there are no compact $S$-spaces}, \emph{there are no first countable $L$-spaces}, \emph{all Aronszajn trees are special}, etc.  In \cite{KT} we showed that the failure of Souslin's Hypothesis was consistent with MA$_{\omega_1}$($\sigma$-centred); in models of form e.g. PFA$(S)[S]$ one obtains most of the Souslin-type consequences of PFA, but many of the combinatorial ones fail, indeed $\mathfrak{p} = \aleph_1$.

From the point of view of a topologist, the name of the game is to use consequences of MA$_{\omega_1}$, PFA, or MM proven -- usually with more difficult proofs -- from MA$_{\omega_1}(S)[S]$, PFA$(S)[S]$, or MM$(S)[S]$, and combine them with consequences failing under MA$_{\omega_1}$, PFA, MM, but holding in models of PFA$(S)[S]$, etc.  More particularly, consequences such as normality implying collectionwise Hausdorffness for certain spaces -- consequences of $V = L$ -- have been shown to hold in some of these models.  It is not yet clear what other kinds of useful consequences of $V = L$ might hold in these models.  The collectionwise Hausdorff ones have been particularly fruitful; they are actually of more than topological applicability, since they can be translated into uniformization of ladder systems or freeness of Whitehead groups -- see Larson-Tall \cite{LT1}.

\section{Some Consequences}
To hold the reader's attention, let us mention some important results obtained by this method.  No other method is known to prove the consistency of the conclusions.

\begin{enumerate}
\item{
\emph{MA$_{\omega_1}(S)[S]$ implies that if $X$ is a compact space with $X^2$ hereditarily normal, then $X$ is metrizable}.  Kat\v{e}tov had proved this in ZFC for $X^3$ hereditarily normal and naturally asked about $X^2$.  50 years later, P. Larson and Todorcevic solved the problem \cite{LTo}.  Consistent counterexamples had earlier been constructed by G. Gruenhage and P. Nyikos \cite{GN}.
}
\item{
\emph{There is a model of MA$_{\omega_1}(S)[S]$ in which locally compact, perfectly normal spaces are paracompact}.  There are many consistent counterexamples, e.g. MA$_{\omega_1}$ implies the Cantor tree over a $Q$-set is a counterexam\-ple; $\diamondsuit$ implies Ostaszewski's space is a counterexample, etc.  The problem of whether it was consistent there are no counterexamples was raised by S. Watson \cite{W}, \cite{W1}.  Larson and Tall \cite{LT1} constructed the required model of PFA$(S)[S]$; A. Dow and Tall managed to drop the large cardinal \cite{DT2}.
}
\item{
\emph{\textup{MA}$_{\omega_1}(S)[S]$ implies every hereditarily normal manifold of dimension $> 1$ is metrizable}.  The problem of the existence of such a model was raised by Nyikos in 1983 \cite{N} and solved by Dow and Tall 30 years later in \cite{DT1}.  Although they don't bother to get MA$_{\omega_1}(S)[S]$, their model could be tweaked to be of that form, with additional consequences of PFA thrown in.
}
\item{
\emph{MM$(S)[S]$ implies locally compact, hereditarily normal spaces are hereditarily paracompact if and only if they do not include a copy of $\omega_1$}.  This was proved by Dow and Tall \cite{DT3}, building on Larson-Tall \cite{LT2}.  Again, the Cantor tree on a $Q$-set and Ostaszewski's space \cite{O} are consistent counterexamples.
}

\item[] Lest the reader think that this method is only of interest for locally compact, hereditarily normal spaces, let us mention:

\item[5.]\textit{
MM$(S)[S]$ implies locally compact, normal, countably tight spaces are paracompact if and only if they do not include a copy of $\omega_1$ and their countable subsets have Lindel\"of closures \cite{DT3}.
}
\item[6.]\textit{
PFA$(S)[S]$ implies countably compact, perfectly normal spaces are compact \cite{DT2}.
}
\end{enumerate}

Of course this last result is less interesting, since it was already known to follow from MA$_{\omega_1}$.  However, the fact that this conclusion may be obtained in a context in which $\mathfrak{p} = \aleph_1$ is of interest. The large cardinal is not needed. 

To avoid having a consequence of MA$_{\omega_1}$, we can combine (2) and (6) to obtain
\begin{itemize}
	\item[6$^\prime$.]\textit{ There is a model of MA$_{\omega_1}(S)[S]$ in which every locally countably compact, perfectly normal space is paracompact.}
\end{itemize}

Again, Ostaszewski's space is a consistent counterexample.

\section{Two models}
Topologists with no knowledge of forcing may wish to skim this section.  Our plan is to present two models of set theory.  The first will be of ``form MA$_{\omega_1}(S)[S]$", i.e. a particular model of MA$_{\omega_1}(S)$ will be forced over by $S$.  The aim of this model is to get as many consequences of PFA as possible, without invoking large cardinals.  The second model will be a model of MM$(S)[S]$, which -- just as for PFA$(S)[S]$ -- is constructed using a supercompact cardinal, but in which certain reflection principles hold which do not follow from PFA$(S)[S]$ \cite{Tp}, \cite{DT3}.  Typically, from MA$_{\omega_1}(S)[S]$, PFA$(S)[S]$, or MM$(S)[S]$ we may only deduce $\aleph_1$-collectionwise Hausdorffness for the kinds of spaces (first countable normal or locally compact normal) we are interested in.  Fortunately, some preliminary forcing before forcing MA$_{\omega_1}(S)[S]$, PFA$(S)[S]$, or MM$(S)[S]$ will enable us to obtain full collectionwise Hausdorffness.  If, for some reason, one would like to avoid that preliminary forcing, one could instead use a reflection principle (\textbf{Axiom R}) following from MM$(S)[S]$ to get results such as

\begin{quotation}
\emph{MM$(S)[S]$ implies locally separable normal first countable spaces are collectionwise Hausdorff.}
\end{quotation}

Since we have no use for this here, we shall not consider it further.

In the MA$_{\omega_1}(S)[S]$ case, the preliminary forcing consists of an Easton extension adjoining $\kappa^+$ Cohen subsets of each regular uncountable cardinal $\kappa$.  This produces $\diamondsuit$ and hence a coherent Souslin tree, and that, after any further $\aleph_2$-chain condition forcing of size $\leq \aleph_2$, normal spaces of character $\leq \aleph_1$ that are $\aleph_1$-collectionwise Hausdorff are collectionwise Hausdorff \cite{Tpp}.

In the PFA$(S)[S]$ or MM$(S)[S]$ cases, before doing the Easton extension, one should first force to make the supercompact cardinal indestructible under countably closed forcing \cite{L}.

In the case where we want to avoid large cardinals, we assume GCH in the ground model and, after our preliminary forcing, iterate $\aleph_2$-p.i.c. posets in a countable support iteration of length $\omega_2$ before forcing with $S$.  By a standard L\"owenheim-Skolem argument, we can get MA$_{\omega_1}(S)[S]$ this way.  We should mention that countable support proper iterations of $S$-preserving proper posets are proper and preserve $S$ \cite{M1}.  Similarly, ``nice" semi-proper iterations of $S$-preserving posets are semi-proper and $S$-preserving \cite{M}.

In the large cardinal avoidance model, we also want to obtain a variety of other consequences of PFA.  This is ad hoc: one must in each case check that a countable support iteration of length $\omega_2$ of $\aleph_2$-p.i.c. posets will establish the desired consequence after forcing with $S$, using some form of $\diamondsuit$ on $\omega_2$ for a L\"owenheim-Skolem argument, and interleaving the iterations for each of the consequences one is interested in.  This is done in detail for two consequences of PFA in \cite{DT1}, and asserted for others in \cite{To}.

\section{Two black boxes}
In this section we present a list of useful propositions which can jointly be obtained without large cardinals in a particular model of MA$_{\omega_1}(S)[S]$ or in a particular model of MM$(S)[S]$ mentioned in the previous section.   We then augment the list with several propositions obtainable from MM$(S)[S]$.  In a number of cases, proofs are not yet available in print or even in preprint form.  We shall note the status of such propositions so that researchers can employ suitable caution.

To make the list less bulky, we shall employ a number of abbreviations detailed below.

\vspace{.2cm}

\noindent
\begin{minipage}[t]{1cm}
	\textbf{CW}
\end{minipage}
\begin{minipage}[t]{12.5cm}
	{First countable normal spaces are collectionwise Hausdorff.}
\end{minipage}

\vspace{.2cm}
\noindent
\begin{minipage}[t]{.9cm}
	\textbf{HL}
\end{minipage}
\begin{minipage}[t]{12.6cm}
	First countable hereditarily Lindel\"of spaces are hereditarily separable.
\end{minipage}

\vspace{.2cm}
\noindent
\begin{minipage}[t]{1.3cm}
	\textbf{M-M}
\end{minipage}
\begin{minipage}[t]{12.2cm}
	Compact countably tight spaces are sequential.
\end{minipage}

\vspace{.2cm}
\noindent
\begin{minipage}[t]{1.5cm}
	\textbf{M-M}${}^{-}$
\end{minipage}
\begin{minipage}[t]{11cm}
	Compact countably tight spaces are sequentially compact.
\end{minipage}

\vspace{.2cm}
\noindent
\begin{minipage}[t]{.7cm}
	$\bsigma$
\end{minipage}
\begin{minipage}[t]{12.8cm}
	Let $X$ be a compact countably tight space.  Let $Y \subseteq X$, $\abs{Y} = \aleph_1$.  Suppose $\{W_\alpha\}_{\alpha < \omega_1}$, $\{N_\alpha\}_{\alpha < \omega_1}$ are open subsets of $X$ such that:
	\begin{enumerate}
	\item{
	$\overline{W}_{\alpha} \subseteq V_\alpha$,
	}
	\item{
	$\abs{V_\alpha \cap Y} \leq \aleph_0$,
	}
	\item{
	$Y \subseteq \bigcup\{W_\alpha : \alpha < \omega_1\}$.
	}
	\end{enumerate}
	Then $Y$ is $\sigma$-closed discrete in $\bigcup\{W_\alpha : \alpha < \omega_1\}$.
\end{minipage}

\vspace{.2cm}
\noindent
\begin{minipage}[t]{1cm}
	$\bsigma^-$
\end{minipage}
\begin{minipage}[t]{12.5cm}
	In a compact, countably tight space, locally countable subspaces of size $\leq \aleph_1$ are $\sigma$-discrete.
\end{minipage}

\vspace{.2cm}
\noindent
\begin{minipage}[t]{3.4cm}
	$\bsigma^-$\textbf{(sequential)}
\end{minipage}
\begin{minipage}[t]{10.1cm}
	In a compact sequential space, locally countable subspaces of size $\leq \aleph_1$ are $\sigma$-discrete.
\end{minipage}

\vspace{.2cm}
\noindent
\begin{minipage}[t]{.9cm}
	$\bsigma^*$
\end{minipage}
\begin{minipage}[t]{12.6cm}
	In a locally compact space of Lindel\"of number $\leq \aleph_1$ not including a perfect pre-image of $\omega_1$, locally countable subspaces of size $\aleph_1$ are $\sigma$-closed discrete.
\end{minipage}

\vspace{.2cm}
\noindent
\begin{minipage}[t]{1.3cm}
	$\PPI^+$
\end{minipage}
\begin{minipage}[t]{12.2cm}
	Every sequentially compact non-compact space (of character $\leq \aleph_1$) includes an uncountable free sequence (a copy of $\omega_1$).
\end{minipage}

\vspace{.2cm}
\noindent
\begin{minipage}[t]{1.3cm}
	$\PPI'$
\end{minipage}
\begin{minipage}[t]{12.2cm}
	Every countably compact, non-compact, first countable space includes a copy of $\omega_1$.
\end{minipage}

\vspace{.2cm}
\noindent
\begin{minipage}[t]{1.2cm}
	$\PPI$
\end{minipage}
\begin{minipage}[t]{12.3cm}
	Every first countable perfect pre-image of $\omega_1$ includes a copy of $\omega_1$.
\end{minipage}

\vspace{.2cm}
\noindent
\begin{minipage}[t]{1.9cm}
	\textbf{CTPPI}
\end{minipage}
\begin{minipage}[t]{11.6cm}
	Every countably tight perfect pre-image of $\omega_1$ includes a copy of $\omega_1$.
\end{minipage}

\vspace{.2cm}

Recall a collection of countable subsets of an uncountable set $X$ is a \emph{$P$-ideal} if
\begin{enumerate}
\item{
$J \subseteq I \in \mc{I}$ implies $J \in \mc{I}$,
}
\item{
Whenever $I_n$, $n < \omega$, $\in \mc{I}$, there is a $J \in \mc{I}$ such that for all $n$, $I_n - J$ is finite.
}
\end{enumerate}

\noindent
\begin{minipage}[t]{2cm}
	\textbf{PID($\aleph_1$)}
\end{minipage}
\begin{minipage}[t]{11.5cm}
	Let $\abs{X} = \aleph_1$, let $\mc{I}$ be a $P$-ideal on $X$.  Then either:
	\begin{enumerate}
	\item{
	there is an uncountable $A \subseteq X$ such that for every countable $B \subseteq A$, $B \in \mc{I}$,
	}
	\item[]{or}
	\item{
	$X = \bigcup_{n < \omega}B_n$ such that for every $n < \omega$ and every $I \in \mc{I}$, $B_n \cap I$ is finite.
	}
	\end{enumerate}
\end{minipage}

\vspace{.2cm}
\noindent
\begin{minipage}[t]{1cm}
	$\mathbf{P}_{22}$
\end{minipage}
\begin{minipage}[t]{12.5cm}
	For every stationary $S \subseteq \omega_1$, for every $P$-ideal $\mc{I}$ on $S$, either:
	\begin{enumerate}
	\item{there is a stationary $A \subseteq S$ such that for every countable $B \subseteq A$, $B \in \mc{I}$,
	}
	\item[]{or}
	\item{
	there is a stationary $A \subseteq S$ such that for every countable $B \subseteq A$ and every $I \in \mc{I}$, $B \cap I$ is finite.
	}
	\end{enumerate}
\end{minipage}

\vspace{.2cm}
\noindent
\begin{minipage}[t]{2.1cm}
\textbf{PID$(\aleph_1)^+$}\\
\\
\end{minipage}
\begin{minipage}[t]{11.4cm}
	For every $\aleph_1$-generated ideal $\mathcal{I}$ of subsets of $\omega_1$, either there is a closed unbounded $C\subseteq \omega_1$ such that $[C]^{\aleph_0}\subseteq\mathcal{I}^\perp$, or there is a stationary $D\subseteq{\omega_1}$ such that $[D]^{\aleph_0}\subseteq\mathcal{I}^{\perp\perp}$.
\end{minipage}

\vspace{.5cm}
Recall an ideal $\mathcal{I}$ is \emph{$\kappa$-generated} if there is a $\mathcal{G}\subseteq\mathcal{I}$ of size $\kappa$ such that each member of $\mathcal{I}$ is included in some member of $\mathcal{G}$. For an ideal $\mathcal{I}\subseteq[\omega_1]^{\aleph_0}$, $\mathcal{I}^\perp=\{A\in[\omega_1]^{\aleph_0}:\text{ for all }I\in\mathcal{I},A\cap I\text{ is finite}\}$.

\vspace{.5cm}
\noindent
\begin{minipage}[t]{1.1cm}
\textbf{SAT}
\end{minipage}
\begin{minipage}[t]{12.4cm}
Every Aronszajn tree is special.
\end{minipage}
\vspace{.2cm}

\noindent
\begin{minipage}[t]{1.4cm}
\textbf{OGA}
\end{minipage}
\begin{minipage}[t]{12.3cm}
The Open Graph Axiom (also called the ``Open Coloring Axiom").
If $X\subseteq\mathbb{R}$ and $[X]^2=K_0\cup K_1$ is a partition and $K_0$ is open in $[X]^2$, then either there is an uncountable $Y\subseteq X$ such that $[Y]^2\subseteq K_0$, or else $X=\bigcup_{n<\omega}Y_n$, where each $[Y_n]^2\subseteq K_1$.
\end{minipage}
\vspace{0.5cm}

A procedure for a ``front-loaded" forcing that will establish \textbf{CW} in models of the sort we are considering is given in \cite{LT1}.  A simplified version for the case when no large cardinals are present is given in \cite{DT2}.  \textbf{HL} is proved in a model of form MA$_{\omega_1}(S)[S]$ in \cite{LTo}.  Since countable chain condition partial orders are proper and indeed those of size $\aleph_1$ satisfy the $\aleph_2$-p.i.c., \textbf{HL} will also hold in the other models of MA$_{\omega_1}(S)[S]$ we consider.  

\textbf{M-M} was proved from PFA$(S)[S]$ by Todorcevic in \cite{To}.  Apparently there is a fixable gap in the proof.  \textbf{M-M${}^{-}$} follows.  A stand-alone proof is in \cite{D}.

There are a number of other results related to \textbf{M-M} proved from PFA$(S)[S]$ in \cite{To}. Todorcevic asserts that none of these require large cardinals. We list some of them, using his numbering:
\begin{itemize}
	\item[8.5] Every countably tight compact space has a point of countable character.
	
	\item[8.6] Every non-Lindel\"of subspace of a countably tight compact space has an uncountable discrete subspace.
	
	\item[10.3] If a compact space includes a nonseparable subspace with no uncountable discrete subspace, then it maps onto $[0,1]^{\aleph_1}$.
	
	\item[10.2] If $Y$ is a nonseparable subspace of a regular space $X$, then either $Y$ includes an uncountable discrete subspace or a subset $Z$ such that the closure of $Z$ in $X$ has no point of countable $\pi$-character.
	
	\item[10.6] If $X$ is a ccc $T_5$ compact space, then $X$ is hereditarily Lindel\"of and hereditarily separable. 
\end{itemize}

$\bsigma$ and hence $\bsigma^-$, and $\bsigma^*$ were proved from PFA$(S)[S]$ in \cite{FTT}, assuming \textbf{M-M}.  Without that assumption the proof yields $\bsigma$ for compact sequential spaces, and hence $\bsigma^-$\textbf{(sequential)}.  

$\PPI^+$, and hence $\PPI'$ and $\PPI$ were proven from PFA$(S)[S]$ in \cite{DT1}.  It was also shown how to modify the proof so as to avoid large cardinals.  Those modifications enable the construction of a model in which these propositions hold as well as \textbf{HL}, $\bsigma^-$\textbf{(sequential)}, and $\mathbf{P}_{22}$.  That \textbf{M-M} and \textbf{PID($\aleph_1$)} hold in such models is asserted in \cite{To}.  \textbf{PID$(\aleph_1)$} can be obtained by a proof similar to that for $\mathbf{P}_{22}$ -- see \cite{DT1}; a correct proof for \textbf{M-M} can be modified as in \cite{D2} and \cite{DT1} for such models.

\textbf{CTPPI} was obtained by Alan Dow via a more complicated version of the proof for $\PPI^+$.  This appears in \cite{DT2}.

\textbf{PID$(\aleph_1)^+$} was incorrectly stated as 6.4 in \cite{To}, where it was asserted it could be obtained without large cardinals.  Dow has verified that for the version stated here.

\textbf{OGA} was obtained from PFA$(S)[S]$ in \cite{To}.

\begin{thm}
	There is a model of {\normalfont MA$_{\omega_1}(S)[S]$} in which all of the consequences listed above in this section hold.
\end{thm}
\begin{proof}
	An amalgamation of \cite{DT1}, \cite{DT2}, \cite{LTo} and \cite{To}.
\end{proof}

We should also mention that 

\begin{thm}[ \cite{M},\cite{W}]
	MA$_{\omega_1}(S)[S]$ implies $2^{\aleph_0} = 2^{\aleph_1} = \aleph_2=\mathfrak{b}=\mathfrak{q}$.
\end{thm}

With the assumption of the existence of a supercompact cardinal, stronger results may be obtained. Rather than using {\normalfont PFA}$(S)[S]$ as in previous papers, we shall move directly to {\normalfont MM}$(S)[S]$ for maximal power. We introduce more abbreviations:
\begin{itemize}
	\item[\textbf{PID}]{
		PID$(\aleph_1)$ without the cardinality restriction.
		}
\end{itemize}

\noindent
\begin{minipage}[t]{1.2cm}
$\textbf{PID}^+$
\end{minipage}
\begin{minipage}[t]{12cm}
	PID$(\aleph_1)^+$ without the cardinality restriction.
\end{minipage}

\begin{thm}
	\textup{MM}$(S)[S]$ implies \textup{\textbf{PID}} and \textup{\textbf{PID$^+$}}.
\end{thm}
\noindent
PFA$(S)[S]$ will do. Similar remarks apply as for the restricted versions mentioned earlier.

\begin{defn}
	$\Gamma\subseteq[X]^{<\kappa}$ is \textbf{tight} if whenever $\{C_\alpha:\alpha<\delta\}$ is an increasing sequence from $\Gamma$ and $\omega<\text{\normalfont cf}(\delta)<\kappa$, $\bigcup\{C_\alpha:\alpha<\beta\}\in\Gamma$.
\end{defn}

\noindent
\begin{minipage}[t]{2cm}
\textbf{Axiom R}
\end{minipage}
\begin{minipage}[t]{11.5cm}
If $\Sigma\subseteq[X]^{<\omega_1}$ is stationary and $\Gamma\subseteq[X]^{<\omega_2}$ is tight and unbounded, then there is a $Y\in\Gamma$ such that $\mathcal{P}(Y)\cap\Sigma$ is stationary in $[Y] ^{<\omega_1}$.
\end{minipage}

\vspace{.5cm}

\textbf{Axiom R} was introduced by W. Fleissner in \cite{F}. An important consequence is:

\begin{prop}[ {\cite{B}}]
	\normalfont{\textbf{ Axiom R}} implies that if $X$ is locally Lindel\"of, regular, and countably tight, then if it is not paracompact, it has a clopen non-paracompact subspace, provided closures of Lindel\"of subspaces are Lindel\"of.
\end{prop}

\noindent
\textbf{Axiom R} can be used with the following consequence of \textbf{PID} in order to prove many results about when locally compact normal spaces are paracompact \cite{T2}.

\vspace{.2cm}
\noindent
\begin{minipage}[t]{1.9cm}
	\textbf{ENT} \cite{EN}
\end{minipage}
\begin{minipage}[t]{11.6cm}
	If $X$ is locally compact, then either
			\begin{enumerate}
				\item $X$ is the union of countably many $\omega$-bounded subspaces, or
				\item $X$ has an uncountable closed discrete subspace, or
				\item $X$ has a separable subspace with non-Lindel\"of closure.
			\end{enumerate}
			
\end{minipage}
\newline
\newline

Recall a space is \textbf{$\omega$-bounded} if countable sets have compact closures. There are a number of useful variations of \textbf{ENT}; see \cite{EN} and \cite{T2}.

Another useful reflection principle that is a consequence of \textbf{Axiom R} and hence (as we shall see) holds under MM$(S)[S]$ is \textbf{FRP} \cite{FJSSU}.
\par\xdef\tpd{\the\prevdepth}

\begin{defn}
	The \textbf{non-stationary ideal on $\omega_1$} is the $\sigma$-ideal \textbf{{\textup{NS}}$_{\omega_1}$} of non-stationary subsets of $\omega_1$.
\end{defn}

\noindent
\begin{minipage}{1.6cm}
\textbf{NSSAT}\\
\\
\end{minipage}
\begin{minipage}{11.9cm}
	is the assertion that NS$_{\omega_1}$ is \textbf{saturated}, i.\! e. there does not exist a family $\{A_\alpha:\alpha<\omega_2\}$ of stationary subsets of $\omega_1$ such that $A_\alpha\cap A_\beta$ is non-stationary, for every $\alpha\neq\beta\in\omega_2$.
\end{minipage}

\prevdepth\tpd
\vspace{.5cm}

\noindent
\begin{minipage}[t]{1cm}
\textbf{SCC}\\
\\
\end{minipage}
\begin{minipage}[t]{12.5cm}
	Strong Chang Conjecture. Let $\lambda>2^{\aleph_2}$ be a regular cardinal. Let $H(\lambda)$ be the collection of hereditarily $<\lambda$ sets. Let $M^*$ be an expansion of $\langle H_\lambda,\in\rangle$. Let $N\prec M^*$ (i.\,e. $N$ is an elementary submodel of $M^*$) be countable. Then there is an $N'$ such that $N\prec N'\prec M^*$, $N'\cap \omega_1=N\cap\omega_1$, and $|N\cap\omega_2|=\aleph_1$.
\end{minipage}

\vspace{.5cm}

\noindent
\begin{minipage}[t]{1cm}
\textbf{SRP}\\
\\
\end{minipage}
\begin{minipage}[t]{12.5cm}
	Strong Reflection Principle \cite{T,FJ,Woo}. Suppose $\lambda\geq\aleph_2$ and $\mathfrak{Z}\subseteq\mathcal{P}_{\omega_1}(\lambda)$ and that for each stationary $T\subseteq\omega_1$,
		\begin{equation*}
			\{\sigma\in\mathfrak{Z}:\sigma\cap\omega_1\in T\}
		\end{equation*}
		is stationary in $\mathcal{P}_{\omega_1}(\lambda)$. Then for all $X\subseteq\lambda$ of cardinality $\aleph_1$, there exists $Y\subseteq\lambda$ such that:
		\begin{itemize}
			\item[(a)] $X\subseteq Y$ and $|Y|=\aleph_1$;
			\par\xdef\tpd{\the\prevdepth}
			\item[(b)] $\mathfrak{Z}\cap\mathcal{P}_{\omega_1}(Y)$ contains a set which is closed unbounded in $\mathcal{P}_{\omega_1}(Y)$.
		\end{itemize}
\end{minipage}

\prevdepth\tpd
\vspace{.5cm}

\begin{thm}[{ \cite{T,Woo}}]
	\textbf{\textup{SRP}} implies $2^{\aleph_0}=2^{\aleph_1}=\aleph_2$.
\end{thm}

\begin{thm}[{ \cite{T}}]
	\textbf{\textup{SRP}} implies \textbf{\textup{NSSAT}}.
\end{thm}

\begin{thm}[ \cite{FMS}]
	\textbf{\textup{SRP}} implies \textbf{\textup{SCC}}.
\end{thm}

\begin{thm}[ {\cite{DT3}}]
	\textbf{\textup{SRP}} implies \textbf{\textup{Axiom R}}.
\end{thm}

It is conjectured that MM$(S)[S]$ implies \textbf{SRP}; at present, we just know:

\begin{thm}[ {\cite{M}}]
	\textup{MM}$(S)$ implies \textbf{\textup{SRP}}.
\end{thm}

We also have:

\begin{thm}[ {\cite{DT3}}]
	\textup{MM}$(S)[S]$ implies \textbf{\textup{Axiom R}}.
\end{thm}

\noindent
\begin{minipage}[t]{1cm}
\textbf{LCN}\\
\end{minipage}
\begin{minipage}[t]{12.5cm}
	Every locally compact normal space is collectionwise Hausdorff.
\end{minipage}
\begin{cor}[ \cite{DT3}]
	There is a model of \textbf{\textup{MM}$(S)[S]$} in which \textbf{\textup{LCN}} holds.
\end{cor}

\textbf{LCN} was claimed in the model of \cite{LT1} in \cite{T1}, but the argument was flawed. In \cite{DT3}, the additional tools of \textbf{NSSAT} and \textbf{SCC} are deployed in order to prove \textbf{LCN}.

We also have

\begin{thm}
	\textup{MM}$(S)[S]$ implies \textbf{\textup{NSSAT}}.
\end{thm}

\begin{proof}
	According to \cite{FMS}, \textbf{NSSAT} is preserved by ccc forcing.
\end{proof}

\begin{thm}
	\textup{MM}$(S)[S]$ implies \textbf{\textup{SCC}}.
\end{thm}

\begin{proof}
	Alan Dow (unpublished) has shown that forcing with a Souslin tree preserves \textbf{SCC}.
\end{proof}

Ideally, one would like to be able to deduce most of the topological consequences of MM$(S)[S]$ from a couple of combinatorial propositions, one following from MM$(S)$ and the other being some generalization of ``first countable normal implies collectionwise Hausdorff''. Getting the second one is plausible, since the quoted statement is known to be equivalent to a combinatorial proposition about uniformization of ladder systems -- see e.g. \cite{LT1}. Important progress along these lines was made by Todorcevic \cite{To} who proved hereditarily separable subspaces of compact sequential spaces are hereditarily Lindel\"of, assuming \textbf{PID} and $\mathfrak{b}=\aleph_2$ and by Alan Dow (unpublished), who derived $\bsigma^-$\textbf{(sequential)} from $\textbf{PID}^+$, but how much more one can achieve remains unclear.

\nocite{*}
\bibliographystyle{acm}
\bibliography{masses}

{\rm Franklin D. Tall, Department of Mathematics, University of Toronto, Toronto, Ontario M5S 2E4, CANADA}

{\it e-mail address:} {\rm tall@math.utoronto.ca}

\end{document}